\documentclass[12pt]{article}

\usepackage[utf8]{inputenc}
\usepackage[a4paper,nomarginpar]{geometry}
\usepackage[english]{babel}
\usepackage{enumitem}
\usepackage{amsmath,amsfonts,amssymb,amsthm,amscd}
\usepackage{tikz,xcolor}
\usepackage{mathrsfs}
\usepackage[all]{xy}

\allowdisplaybreaks

\theoremstyle{plain}
\newtheorem{theorem}{Theorem}
\newtheorem{lemma}[theorem]{Lemma}
\newtheorem{corollary}[theorem]{Corollary}

\theoremstyle{definition}
\newtheorem{definition}[theorem]{Definition}
\newtheorem{remark}[theorem]{Remark}
\newtheorem{example}[theorem]{Example}

\newcommand{\N}{\mathbb{N}} 
\newcommand{\Z}{\mathbb{Z}} 
\newcommand{\R}{\mathbb{R}} 
\newcommand{\C}{\mathbb{C}} 
\newcommand{\K}{\mathbb{K}} 

\newcommand{\cA}{\mathcal{A}}
\newcommand{\cB}{\mathcal{B}}
\newcommand{\cC}{\mathcal{C}}
\newcommand{\cD}{\mathcal{D}}

\newcommand{\cO}{\mathcal{O}}
\newcommand{\cP}{\mathcal{P}}

\newcommand{\cS}{\mathcal{S}}

\newcommand{\fF}{\mathfrak{F}}
\newcommand{\fM}{\mathfrak{M}}
\newcommand{\fS}{\mathfrak{S}}

\newcommand{\card}{\operatorname{card}}

\newcommand{\ldens}{\operatorname{\underline{dens}}}

\newcommand{\eps}{\varepsilon}

\newcommand{\wt}{\widetilde}

\DeclareRobustCommand{\rchi}{{\mathpalette\irchi\relax}}
\newcommand{\irchi}[2]{\raisebox{\depth}{$#1\chi$}}

\title{On the dynamics of weighted composition operators II}

\author{Nilson C. Bernardes Jr., Antonio Bonilla, Jo\~ao V. A. Pinto}

\date{}

\begin{document}

\maketitle

\begin{abstract}
We establish complete characterizations of various notions of expansivity for weighted composition operators 
on a very general class of locally convex spaces of continuous functions.
This class includes several classical classes of continuous function spaces, 
such as the Banach spaces $C_0(X)$ of continuous scalar-valued functions vanishing at infinity on a Hausdorff locally compact space $X$,
endowed with the sup norm, 
and the locally convex spaces $C(X)_c$ of continuous scalar-valued functions on a completely regular space $X$,
endowed with the compact-open topology.
We also obtain complete characterizations of various notions of expansivity for weighted composition operators on $L^p(\mu)$ spaces,
thereby complementing and extending previously known results in the unweighted case.
Finally, we establish a conjugation between weighted and unweighted composition operators in the case of dissipative systems 
on $L^p(\mu)$ spaces and apply it to the study of several dynamical properties.
\end{abstract}




\section{Introduction}

This work is dedicated to the study of the dynamics of weighted composition operators 
\[
C_{w,f}(\varphi) = w \cdot (\varphi \circ f)
\]
on certain general classes of function spaces and can be considered as a continuation of our recent article~\cite{BerBonPin26}.
Due to the important role played by weighted composition operators in operator theory and its applications, the dynamics of these operators has been extensively investigated by various authors in different contexts over the last few decades.
We refer the reader to the recent articles \cite{CAlvJHen25,AseJorKal25,BerBonPin26,NBerFVas25,KChe24,DGomKGro25},
where many additional references can be found.

Our goal is to investigate various notions of expansivity for weighted composition operators on a very general class of 
continuous function spaces and on $L^p(\mu)$ spaces, as well as several dynamical properties of 
weighted composition operators on $L^p(\mu)$ spaces in the dissipative case.

Given a metric space $M$ with metric $d$, recall that a homeomorphism $h : M \to M$ is said to be {\em expansive} if it admits an {\em expansivity constant} $c > 0$, which means that for any pair $x,y$ of distinct points in $M$, there exists $n \in \Z$ such that
\[
d(h^n(x),h^n(y)) > c.
\]
This concept was introduced by Utz~\cite{WUtz50} and is one of the fundamental concepts in the modern theory of dynamical systems.
In the setting of linear dynamics, Eisenberg~\cite{MEis66} proved that an invertible operator on $\C^n$ is expansive if and only if it has no eigenvalue on the unit circle. 
In the early 1970s, Eisenberg and Hedlund~\cite{MEisJHed70} and Hedlund~\cite{JHed71} studied expansivity for operators on Banach spaces. 
They introduced the notion of uniform expansivity and proved that an invertible operator on a Banach space is uniformly expansive if and only if its approximate point spectrum does not intersect the unit circle. 
In 2000, Mazur~\cite{MMaz00} proved that an invertible normal operator $T$ on a Hilbert space is expansive if and only if $T^*T$ has no eigenvalue on the unit circle. 
This subject was revisited by Bernardes et al.~\cite{BerCirDarMesPuj18} in 2018, where many results were obtained, including complete characterizations of the various notions of expansivity for weighted shifts on classical Banach sequence spaces.
Since then, several works on expansivity for operators have been developed and interesting results have been obtained.
For instance, it was proved that the equalities
\begin{align*}
\text{hyperbolicity } &= \text{ expansivity } + \text{ shadowing}\\
  &= \text{ average expansivity } + \text{ average shadowing}
\end{align*}
hold for invertible operators on any Banach space \cite{AlvBerMes21,NBerAMes21,CirGolPuj21}.
The notions of expansivity have been extended to operators on arbitrary Hausdorff locally convex spaces in the recent papers \cite{BerCarDarFavPer25,BerMarRod}, where investigations on these concepts in this much more general setting have been initiated.
Additional references on expansivity in linear dynamics include 
\cite{CAlvJHen25,BagDasOsaPod25,AniDarMai22,KLeeCMor22,MMai22}.

Let us now describe the organization of the article.

In Section~\ref{S-Preliminaries}, we will present the necessary prerequisites for the development of this work.

In Section~\ref{S-ExpCFS}, we will establish complete characterizations of various notions of expansivity 
for weighted composition operators on a very general class of locally convex spaces of continuous functions, 
namely, the spaces $C_\fS(X;Y)$ of all continuous functions $\varphi$ from a completely regular space $X$ 
into a Hausdorff locally convex space $Y$ such that $\varphi(B)$ is a bounded
subset of $Y$ for each $B \in \fS$ ($\fS$ being a given bornology on $X$), endowed with the topology of $\fS$-convergence.
In fact, our characterizations are also valid for certain subspaces of $C_\fS(X;Y)$ that we call {\em admissible}.
Applications to several classical classes of continuous function spaces, including the Banach spaces $C_0(X)$ and the locally convex spaces $C(X)_c$, will be presented.
As specific examples, we will consider weighted shifts and weighted translation operators.

In Section~\ref{S-ExpLp}, we will establish complete characterizations of various notions of expansivity for weighted composition operators on $L^p(\mu)$ spaces, thereby complementing and extending previous characterizations due to Maiuriello \cite{MMai22} in the unweighted case. We will illustrate the general results by presenting applications to weighted shifts and weighted translation operators.

In Section~\ref{S-Dissipative}, we will study the dynamics of weighted composition operators on $L^p(\mu)$ spaces in the dissipative case.
In the unweighted case, dissipative composition operators have been studied in the recent papers
\cite{BerBonPin26, AniDarMai21, AniDarMai22, EAniMMai23, AniMaiSeo24, UDarBPir21, DGomKGro25, MMai22, MMai23}.
These operators behave very much like weighted shifts. 
In fact, it was shown in \cite{AniDarMai22} that any dissipative composition operator 
satisfying a certain technical condition, called {\em bounded distortion}, has a certain specific weighted shift as a factor
and for many dynamical properties, the composition operator has the property if and only if so does the associated weighted shift.
For this reason, these composition operators were called {\em shift-like} in \cite{AniDarMai22}.
It is well-known that weighted shifts on $\ell^p$ spaces are conjugate 
to unweighted shifts on weighted $\ell^p$ spaces \cite[Section~4.1]{KGroAPer11}. 
In a certain sense, one can transfer the weights from the shift to the space.
Our main goal in Section~\ref{S-Dissipative} is to prove that something similar holds for composition operators 
on $L^p(\mu)$ spaces in the dissipative case.
More precisely, we will prove that any dissipative weighted composition operator $C_{w,f}$ on a space $L^p(\mu)$ 
is conjugate to an unweighted composition operator $C_f$ on a space $L^p(\nu)$ by means of an isometric isomorphism,
where the measure $\nu$ is given by a specific formula involving the measure $\mu$ and the weight function $w$.
This allows us to transfer known results about dissipative composition operators to the weighted case, and we will illustrate 
this process with several dynamical properties, including frequent hypercyclicity, Devaney chaos and the shadowing property.


\section{Preliminaries}\label{S-Preliminaries}

Throughout the article, $\K$ denotes either the field $\R$ of real numbers or the field $\C$ of complex numbers,
$\Z$ denotes the ring of integers, $\N$ denotes the set of all positive integers, and $\N_0 = \N \cup \{0\}$.

Given a seminorm $\|\cdot\|$ on a vector space $Y$, the {\em unit sphere} of $\|\cdot\|$ is defined by
\[
S_{\|\cdot\|} = \{y \in Y : \|y\| = 1\}.
\]
If $Y$ is a normed space with norm $\|\cdot\|$, then we also write $S_Y$ instead of $S_{\|\cdot\|}$.

Let $Y$ be a Hausdorff locally convex space over $\K$.
By an {\em operator} on $Y$, we mean a continuous linear map $T: Y \to Y$.
Such an operator is said to be {\em invertible} if it has a continuous inverse.
We denote by $L(Y)$ (resp.\ $GL(Y)$) the set of all operators (resp.\ invertible operators) on $Y$.
Let $(\|\cdot\|_\alpha)_{\alpha \in I}$ be a family of seminorms that induces the topology of $Y$.
Recall that $T \in GL(Y)$ is said to be {\em topologically expansive} if, for each $y \in Y \backslash \{0\}$, 
there exists $\alpha \in I$ such that 
\[
\sup_{n \in \Z} \|T^n y\|_\alpha = \infty.
\]
Moreover, $T$ is said to be {\em uniformly topologically expansive} if, for each $\alpha \in I$, there exists $\beta \in I$ such that 
we can write $S_{\|\cdot\|_\alpha} = S^+_\alpha \cup S^-_\alpha$, where
\[
\|T^n y\|_\beta \to \infty \text{ uniformly on } S^+_\alpha \text{ as } n \to \infty,
\]
\[
\|T^{-n} y\|_\beta \to \infty \text{ uniformly on } S^-_\alpha \text{ as } n \to \infty.
\]

These concepts of expansivity were introduced in \cite{BerCarDarFavPer25} as generalizations of the classical notions 
of expansivity and uniform expansivity for invertible operators on Banach spaces \cite{MEisJHed70}.
The word ``topologically'' was used to distinguish the above notion of expansivity from the usual one in the metric space setting, 
because they may differ in the case of a metrizable locally convex space (or even a Fr\'echet space) endowed with a compatible invariant 
metric \cite[Example~32]{BerCarDarFavPer25}, although they always coincide in normed spaces.
However, as in \cite{BerMarRod}, we will omit the word ``topologically'' 
and simply write ``expansive'' and ``uniformly expansive'' for the above notions.

A concept of average expansivity for invertible operators on Banach spaces was introduced and investigated in \cite{AlvBerMes21}.
It was extended to the context of Hausdorff locally convex spaces in the recent paper \cite{BerMarRod}, by means of the following
definition: $T \in GL(Y)$ is said to be {\em average expansive} if, for each $y \in Y \backslash \{0\}$, there exists $\alpha \in I$ such that
\[
\sup_{n \in \N} \Big(\frac{1}{2n+1} \sum_{j=-n}^n \|T^j y\|_\alpha \Big) = \infty.
\]

It was observed in \cite[Proposition~2]{BerMarRod} that
\[
\text{uniform expansivity} \ \ \Longrightarrow \ \ \text{average expansivity} \ \ \Longrightarrow \ \ \text{expansivity}.
\]
However, the converses are false even in the Hilbert space setting, as it was shown in \cite[Example~16]{AlvBerMes21}
and \cite[Example~17]{AlvBerMes21}.

For operators that are not necessarily invertible, we have the following variations of the above notions of expansivity
\cite{BerCarDarFavPer25,BerMarRod}: $T \in L(Y)$ is said to be
\begin{itemize}
\item {\em positively expansive} if, for each $y \in Y \backslash \{0\}$, there exists $\alpha \in I$ such that 
  \[
  \sup_{n \in \N} \|T^n x\|_\alpha = \infty.
  \]
\item {\em average positively expansive} if, for each $y \in Y \backslash \{0\}$, there exists $\alpha \in I$ such that
\[
\sup_{n \in \N} \Big(\frac{1}{n} \sum_{j=0}^{n-1} \|T^j y\|_\alpha \Big) = \infty.
\]
\item {\em uniformly positively expansive} if, for each $\alpha \in I$, there exists $\beta \in I$ such that
\[
\|T^n x\|_\beta \to \infty \text{ uniformly on } S_{\|\cdot\|_\alpha} \text{ as } n \to \infty.
\]
\end{itemize}

All the above notions of expansivity do not depend on the choice of the family of seminorms inducing the topology of $Y$,
so they are completely determined by the topology of $Y$.


\section{Expansive weighted composition operators on locally convex spaces of continuous functions}\label{S-ExpCFS}

Throughout this section, we fix a completely regular space $X$ and a bornology $\fS$ on $X$, that is, 
a collection of subsets of $X$ that contains the singletons and is stable under the passage to subsets and the formation of finite unions.
We also fix a Hausdorff locally convex space $Y$ over~$\K$, $Y \neq \{0\}$, 
and a family $(\|\cdot\|_\alpha)_{\alpha \in I}$ of seminorms that induces its topology.
$C(X;Y)$ denotes the vector space over $\K$ of all continuous functions from $X$ into $Y$, and $C_\fS(X;Y)$ is the vector subspace 
of $C(X;Y)$ consisting of all $\varphi \in C(X;Y)$ such that $\varphi(B)$ is a bounded subset of $Y$ for each $B \in \fS$.
We consider $C_\fS(X;Y)$ endowed with the topology of $\fS$-convergence (i.e., the topology of uniform convergence on the sets of $\fS$),
which is the Hausdorff locally convex topology induced by the family of seminorms given by
\[
\|\varphi\|_{B,\alpha} = \sup_{x \in B} \|\varphi(x)\|_\alpha \ \ \ (\varphi \in C_\fS(X;Y) \text{ and } (B,\alpha) \in \fS \times I).
\]
Note that if $\cB$ is a base of the bornology $\fS$ (i.e., $\cB \subset \fS$ and each element of $\fS$ is contained in an element of $\cB$), then the subfamily $(\|\cdot\|_{B,\alpha})_{(B,\alpha) \in \cB \times I}$ also induces the topology of $C_\fS(X;Y)$.
For each collection $\mathscr{C}$ of subsets of $X$ and each $B \subset X$, we define
\[
\mathscr{C}_B = \{A \in \mathscr{C} : A \cap B \neq \varnothing\}.
\]
For each $B \subset X$ and each bounded function $\phi : B \to \K$, we define
\[
\|\phi\|_B = \sup_{x \in B} |\phi(x)|,
\]
where we consider this supremum to be $0$ if $B = \varnothing$.
We also fix a homeomorphism $f : X \to X$ satisfying $f(\fS) = \fS$ and a zero-free continuous function $w : X \to \K$ 
such that both $w$ and $1/w$ are bounded on each set $B \in \fS$.
It is easy to check that the {\em weighted composition operator}
\[
C_{w,f}(\varphi) = w \cdot (\varphi \circ f)
\]
is a well-defined invertible operator on $C_\fS(X;Y)$ with
\[
(C_{w,f})^{-1} = C_{\wt{w},f^{-1}}, \ \ \text{ where } \wt{w} = \frac{1}{w \circ f^{-1}}\cdot
\]
We associate to $w$ and $f$ the following bilateral sequence of continuous functions from $X$ into $\K$:
$w^{(0)} = 1$, $w^{(1)} = w$, $w^{(-1)} = \wt{w}$ and, for each $n \geq 2$,
\[
w^{(n)} = w \cdot (w \circ f) \cdots (w \circ f^{n-1}) \ \ \text{ and } \ \
w^{(-n)} = \wt{w} \cdot (\wt{w} \circ f^{-1}) \cdots (\wt{w} \circ f^{-(n-1)}).
\]
Finally, we say that a subspace $M$ of $C_\fS(X;Y)$ is {\em admissible} if the following conditions hold:
\begin{itemize}
\item [(C1)] $C_{w,f}(M) = M$;
\item [(C2)] For each point $a \in X$, each neighborhood $O$ of $a$ in $X$ and each vector $y \in Y$, 
  there exists a continuous function $\phi : X \to [0,1]$ such that
  \[
  \phi(a) = 1, \ \ \ \phi(X \backslash O) = \{0\} \ \ \text{ and } \ \ \phi \otimes y \in M,
  \]
  where $\phi \otimes y : X \to Y$ is given by $(\phi \otimes y)(x) = \phi(x)y$.
\end{itemize} 

We now establish the main result of this section, namely, characterizations of the various notions of expansivity 
for invertible weighted composition operators on admissible subspaces of $C_\fS(X;Y)$.
Note that, since the topological space $X$ is completely regular, it is clear that $C_\fS(X;Y)$ is an admissible subspace of itself.

\begin{theorem}\label{ExpCX-T1}
Let $M$ be an admissible subspace of $C_\fS(X;Y)$, $\cB$ a base of the bornology $\fS$ and $\cO$ a base of the topology of $X$,
with $\varnothing \not\in \cB$ and $\varnothing \not\in \cO$. 
For the invertible weighted composition operator $C_{w,f}$ on $M$, the following statements hold:
\begin{itemize}
\item [\rm (a)] $C_{w,f}$ is expansive if and only if, for each $O \in \cO$, there exists $B \in \cB$ such that
    \[
    \sup_{n \in \Z} \|w^{(n)}\|_{B \cap f^{-n}(O)} = \infty.
    \]
\item [\rm (b)] $C_{w,f}$ is average expansive if and only if, for each $O \in \cO$, there exists $B \in \cB$ such that
    \[
    \sup_{n \in \N} \Big(\frac{1}{2n+1} \sum_{j=-n}^n \|w^{(j)}\|_{B \cap f^{-j}(O)}\Big) = \infty.
    \]
\item [\rm (c)] $C_{w,f}$ is uniformly expansive if and only if, for each $A \in \cB$,
    there exists $B \in \cB$ such that we can write $\cO_A = \cO_A^+ \cup \cO_A^-$, where
    \[
    \lim_{n \to \infty} \inf_{O \in \cO_A^+} \|w^{(n)}\|_{B\cap f^{-n}(O)} = \infty \ \ \text{ and } \ \ 
    \lim_{n \to \infty} \inf_{O \in \cO_A^-} \|w^{(-n)}\|_{B\cap f^{n}(O)} = \infty.
    \]
\end{itemize}
\end{theorem}

\begin{proof}
(c): Suppose that $C_{w,f}$ is uniformly expansive. 
Given $(A,\alpha) \in \cB \times I$, there exists $(B,\beta) \in \cB \times I$ such that the unit sphere $S_{\|\cdot\|_{A,\alpha}}$
of the seminorm $\|\cdot\|_{A,\alpha}$ in the space $M$ can be written as
$S_{\|\cdot\|_{A,\alpha}} = S^{+}_{A,\alpha} \cup S^{-}_{A,\alpha}$, where
\begin{equation}\label{ExpCX-Eq1}
\lim_{n \to \infty} \inf_{\varphi \in S^{+}_{A,\alpha}} \|(C_{w,f})^n(\varphi)\|_{B,\beta} = \infty \ \ \text{ and } \ \ 
\lim_{n \to \infty} \inf_{\varphi \in S^{-}_{A,\alpha}} \|(C_{w,f})^{-n}(\varphi)\|_{B,\beta} = \infty.
\end{equation} 
For each $O \in \cO_A$, take a point $x_O \in O \cap A$ and a vector $y_O \in Y$ with $\|y_O\|_\alpha = 1$.
By (C2), there is a continuous function $\phi_O : X \to [0,1]$ such that 
$\phi_O(x_O) = 1$, $\phi_O(X \backslash O) = \{0\}$ and $\varphi_O = \phi_O \otimes y_O \in M$.
Define
\[
\cO_A^+ = \{O \in \cO_A : \varphi_O \in S^{+}_{A,\alpha}\} \ \ \text{ and } \ \
\cO_A^- = \{O \in \cO_A : \varphi_O \in S^{-}_{A,\alpha}\}.
\]
Since $\varphi_O \in S_{\|\cdot\|_{A,\alpha}}$ for all $O \in \cO_A$, we have that
\[
\cO_A = \cO_A^+ \cup \cO_A^-.
\]
For each $O \in \cO_A$ and $n \in \Z$, we have that
\[
\|(C_{w,f})^n(\varphi_O)\|_{B,\beta} = \|w^{(n)} \cdot (\varphi_O \circ f^n)\|_{B \cap f^{-n}(O),\beta} 
  \leq \|w^{(n)}\|_{B \cap f^{-n}(O)} \|y_O\|_\beta.
\]
Hence, for each $n \in \N$,
\begin{equation}\label{ExpCX-Eq2}
\inf_{O \in \cO_A^+} \|w^{(n)}\|_{B\cap f^{-n}(O)} 
  \geq \frac{1}{\|y_O\|_\beta} \inf_{\varphi \in S^{+}_{A,\alpha}} \|(C_{w,f})^n(\varphi)\|_{B,\beta}
\end{equation}
and
\begin{equation}\label{ExpCX-Eq3}
\inf_{O \in \cO_A^-} \|w^{(-n)}\|_{B\cap f^{n}(O)} 
  \geq \frac{1}{\|y_O\|_\beta} \inf_{\varphi \in S^{-}_{A,\alpha}} \|(C_{w,f})^{-n}(\varphi)\|_{B,\beta}.
\end{equation}
By (\ref{ExpCX-Eq1}), (\ref{ExpCX-Eq2}) and (\ref{ExpCX-Eq3}), the limits in (c) are infinite.

Conversely, suppose that the condition in (c) holds. Take $(A,\alpha) \in \cB \times I$.
There exists $B \in \cB$ such that we can write $\cO_A = \cO_A^+ \cup \cO_A^-$, where
\begin{equation}\label{ExpCX-Eq4}
\lim_{n \to \infty} \inf_{O \in \cO_A^+} \|w^{(n)}\|_{B\cap f^{-n}(O)} = \infty \ \ \text{ and } \ \ 
\lim_{n \to \infty} \inf_{O \in \cO_A^-} \|w^{(-n)}\|_{B\cap f^{n}(O)} = \infty.
\end{equation}
For each $\varphi \in S_{\|\cdot\|_{A,\alpha}}$, take $O_\varphi \in \cO_A$ such that
\[
\|\varphi(x)\|_\alpha \geq \frac{1}{2} \ \ \text{ for all } x \in O_\varphi.
\]
Let
\[
S^{+}_{A,\alpha} =\{\varphi \in S_{\|\cdot\|_{A,\alpha}} : O_\varphi \in \cO_A^+\} \ \ \text{ and } \ \
S^{-}_{A,\alpha} =\{\varphi \in S_{\|\cdot\|_{A,\alpha}} : O_\varphi \in \cO_A^-\}.
\]
Clearly, 
\[
S_{\|\cdot\|_{A,\alpha}} = S^{+}_{A,\alpha} \cup S^{-}_{A,\alpha}.
\] 
For each $\varphi \in S_{\|\cdot\|_{A,\alpha}}$ and $n \in \Z$, we have that
\[
\|(C_{w,f})^n(\varphi)\|_{B,\alpha} \geq \|w^{(n)} \cdot (\varphi \circ f^n)\|_{B \cap f^{-n}(O_\varphi),\alpha} 
  \geq \frac{1}{2}\, \|w^{(n)}\|_{B \cap f^{-n}(O_\varphi)}.
\]
Hence, for each $n \in \N$,
\begin{equation}\label{ExpCX-Eq5}
\inf_{\varphi \in S^{+}_{A,\alpha}} \|(C_{w,f})^n(\varphi)\|_{B,\alpha} 
  \geq \frac{1}{2} \inf_{O \in \cO_A^+} \|w^{(n)}\|_{B\cap f^{-n}(O)}
\end{equation}
and
\begin{equation}\label{ExpCX-Eq6}
\inf_{\varphi \in S^{-}_{A,\alpha}} \|(C_{w,f})^{-n}(\varphi)\|_{B,\alpha} 
  \geq \frac{1}{2} \inf_{O \in \cO_A^-} \|w^{(-n)}\|_{B\cap f^n(O)}.
\end{equation}
By (\ref{ExpCX-Eq4}), (\ref{ExpCX-Eq5}) and (\ref{ExpCX-Eq6}),
\[
\lim_{n \to \infty} \inf_{\varphi \in S^{+}_{A,\alpha}} \|(C_{w,f})^n(\varphi)\|_{B,\alpha} = \infty \ \text{ and } \ 
\lim_{n \to \infty} \inf_{\varphi \in S^{-}_{A,\alpha}} \|(C_{w,f})^{-n}(\varphi)\|_{B,\alpha} = \infty,
\]
which proves that $C_{w,f}$ is uniformly expansive.

\smallskip\noindent
(b): Suppose that $C_{w,f}$ is average expansive. Given $O \in \cO$, choose a point $a \in O$ and a vector $y \in Y \backslash \{0\}$.
By (C2), there is a continuous function $\phi : X \to [0,1]$ such that 
$\phi(a) = 1$, $\phi(X \backslash O) = \{0\}$ and $\varphi = \phi \otimes y \in M$.
Take $(B,\alpha) \in \cB \times I$ such that
\begin{equation}\label{ExpCX-Eq7}
\sup_{n \in \N} \Big(\frac{1}{2n+1} \sum_{j=-n}^n \|(C_{w,f})^j(\varphi)\|_{B,\alpha}\Big) = \infty.
\end{equation}
For each $n \in \Z$, we have that
\begin{equation}\label{ExpCX-Eq8}
\|(C_{w,f})^n(\varphi)\|_{B,\alpha} = \|w^{(n)} \cdot (\varphi \circ f^n)\|_{B \cap f^{-n}(O),\alpha} 
  \leq \|w^{(n)}\|_{B \cap f^{-n}(O)} \|y\|_\alpha.
\end{equation}
By (\ref{ExpCX-Eq7}) and (\ref{ExpCX-Eq8}), we obtain
\[
\sup_{n \in \N} \Big(\frac{1}{2n+1} \sum_{j=-n}^n \|w^{(j)}\|_{B \cap f^{-j}(O)}\Big) = \infty,
\]
as it was to be shown.

Conversely, suppose that the condition in (b) holds.
Given $\varphi \in M \backslash \{0\}$, there exist $C > 0$, $\alpha \in I$ and $O \in \cO$ such that
\[
\|\varphi(x)\|_\alpha \geq C \ \ \text{ for all } x \in O.
\]
By hypothesis, there exists $B \in \cB$ such that
\begin{equation}\label{ExpCX-Eq9}
\sup_{n \in \N} \Big(\frac{1}{2n+1} \sum_{j=-n}^n \|w^{(j)}\|_{B \cap f^{-j}(O)}\Big) = \infty.
\end{equation}
For each $n \in \Z$, we have that
\begin{equation}\label{ExpCX-Eq10}
\|(C_{w,f})^n(\varphi)\|_{B,\alpha} \geq \|w^{(n)} \cdot (\varphi \circ f^n)\|_{B \cap f^{-n}(O),\alpha} 
  \geq C\, \|w^{(n)}\|_{B \cap f^{-n}(O)}.
\end{equation}
By (\ref{ExpCX-Eq9}) and (\ref{ExpCX-Eq10}), we obtain
\[
\sup_{n \in \N} \Big(\frac{1}{2n+1} \sum_{j=-n}^n \|(C_{w,f})^j(\varphi)\|_{B,\alpha}\Big) = \infty,
\]
which shows that $C_{w,f}$ is average expansive. 

\smallskip\noindent
(a): The proof is similar to that of (b).
\end{proof}

\begin{example}\label{ExpCX-Ex1}
If $Y = \K$ and $\fS$ is the bornology of all subsets of $X$, then $C_\fS(X;Y)$ is the Banach space $C_b(X)$ of all bounded continuous functions $\varphi : X \to \K$ endowed with the supremum norm
\[
\|\varphi\| = \sup_{x \in X} |\varphi(x)|,
\]
which induces the topology of uniform convergence.
By setting $\cB = \{X\}$ and choosing a base $\cO$ of the topology of $X$ consisting of nonempty open sets, 
Theorem~\ref{ExpCX-T1} tell us that the following characterizations hold for the invertible weighted composition operator 
$C_{w,f}$ on $C_b(X)$:
\begin{itemize}
\item [\rm (a)] $C_{w,f}$ is expansive if and only if
    \[
    \sup_{n \in \Z} \|w^{(n)}\|_{f^{-n}(O)} = \infty \ \text{ for all } O \in \cO.
    \]
\item [\rm (b)] $C_{w,f}$ is average expansive if and only if 
    \[
    \sup_{n \in \N} \Big(\frac{1}{2n+1} \sum_{j=-n}^n \|w^{(j)}\|_{f^{-j}(O)}\Big) = \infty \ \text{ for all } O \in \cO.
    \]
\item [\rm (c)] $C_{w,f}$ is uniformly expansive if and only if $\cO = \cO^+ \cup \cO^-$, where
\[
\lim_{n \to \infty} \inf_{O \in \cO^+} \|w^{(n)}\|_{f^{-n}(O)} = \infty \ \ \text{ and } \ \ 
\lim_{n \to \infty} \inf_{O \in \cO^-} \|w^{(-n)}\|_{f^{n}(O)} = \infty.
\]
\end{itemize}
\end{example}

\begin{example}\label{ExpCX-Ex2}
If $X$ is a locally compact Hausdorff space, then the classical Banach space $C_0(X)$ of all continuous functions 
$\varphi : X \to \K$ that vanish at infinity is an admissible subspace of $C_b(X)$.
Hence, we can regard $C_{w,f}$ as an operator on $C_0(X)$ and the characterizations of the various notions of expansivity
in this case are exactly as in the previous example. 
Moreover, we can choose $\cO$ as a base of relatively compact open sets in $X$.
As a concrete example, we have the {\em bilateral weighted translation operators}
\[
T_w : \varphi \in C_0(\R) \mapsto w(\cdot) \varphi(\cdot + 1) \in C_0(\R),
\]
which correspond to $X = \R$ with its usual topology and $f(x) = x+1$ for all $x \in \R$.
In this particular case, we can take $\cO$ as the countable set of all nonempty open intervals in $\R$ with rational endpoints, for instance.
\end{example}

\begin{example}\label{ExpCX-Ex3}
If $Y = \K$ and $\fS$ is the bornology of all compact subsets of $X$, then $C_\fS(X;Y)$ is the classical locally convex space 
$C(X)_c$ of all continuous functions $\varphi : X \to \K$ endowed with the topology of compact convergence (or compact-open topology).
By choosing any base $\cB$ of the bornology $\fS$ and any base $\cO$ of the topology of $X$
with $\varnothing \not\in \cB$ and $\varnothing \not\in \cO$, 
we have that the characterizations (a), (b) and (c) given in Theorem~\ref{ExpCX-T1} 
hold for the invertible weighted composition operator $C_{w,f}$ on $C(X)_c$.
In particular, if there exists a nonempty open set $W$ in $X$ such that each compact set $K$ in $X$ intersects $f^n(W)$
for only finitely many indices $n \in \Z$, then $C_{w,f}$ cannot be expansive. 
For instance, weighted shifts $B_w$ on the product space $\K^\Z$ and weighted translation operators $T_w$
on the locally convex space $C(\R)_c$ are never expansive. 
\end{example}

\begin{example}\label{ExpCX-Ex4}
If $Y = \K$ and $\fS$ is the bornology of all finite subsets of $X$, then $C_\fS(X;Y)$ is the classical locally convex space 
$C(X)_s$ of all continuous functions $\varphi : X \to \K$ endowed with the topology of pointwise convergence.
By choosing any base $\cO$ of the topology of $X$ consisting of nonempty open sets, it follows from Theorem~\ref{ExpCX-T1} that 
the following characterizations hold for the invertible weighted composition operator $C_{w,f}$ on $C(X)_s$:
\begin{itemize}
\item [\rm (a)] $C_{w,f}$ is expansive if and only if, for each $O \in \cO$, there exists $b \in X$ such that
    \[
    \sup_{n \in J} |w^{(n)}(b)| = \infty,
    \]
    where $J = \{n \in \Z : f^n(b) \in O\}$.
\item [\rm (b)] $C_{w,f}$ is average expansive if and only if, for each $O \in \cO$, there exists $b \in X$ such that
    \[
    \sup_{n \in \N} \Big(\frac{1}{2n+1} \sum_{j \in J_n} |w^{(j)}(b)|\Big) = \infty,
    \]
    where $J_n = \{j \in \{-n,\ldots,n\} : f^n(b) \in O\}$ ($n \in \N$).
\item [\rm (c)] $C_{w,f}$ is uniformly expansive if and only if, for each $a \in X$,
    there exists $B \subset X$ finite such that we can write $\cO_{\{a\}} = \cO_{\{a\}}^+ \cup \cO_{\{a\}}^-$, where
    \[
    \lim_{n \to \infty} \inf_{O \in \cO_{\{a\}}^+} \|w^{(n)}\|_{B \cap f^{-n}(O)} = \infty \ \ \text{ and } \ \ 
    \lim_{n \to \infty} \inf_{O \in \cO_{\{a\}}^-} \|w^{(-n)}\|_{B \cap f^{n}(O)} = \infty.
    \]
\end{itemize}
\end{example}

In the next result, we will see that certain additional hypotheses imply that it is sufficient to test the conditions in (a) and (b) 
of Theorem~\ref{ExpCX-T1} for a smaller collection of sets $O$ in $\cO$.

\begin{corollary}\label{ExpCX-C1}
Let $M$, $\cB$ and $\cO$ be as in Theorem~\ref{ExpCX-T1}. Suppose that both $w$ and $1/w$ are bounded on $X$ 
and that  there exists a nonempty open set $W$ in $X$ such that $\bigcup_{n \in \Z} f^n(W)$ is dense in $X$. 
For the invertible weighted composition operator $C_{w,f}$ on $M$, the following statements hold:
\begin{itemize}
\item [\rm (a)] $C_{w,f}$ is expansive if and only if, for each $O \in \cO$ with $O \subset W$, there exists $B \in \cB$ such that
    \[
    \sup_{n \in \Z} \|w^{(n)}\|_{B \cap f^{-n}(O)} = \infty.
    \]
\item [\rm (b)] $C_{w,f}$ is average expansive if and only if, for each $O \in \cO$ with $O \subset W$, there exists $B \in \cB$ such that
    \[
    \sup_{n \in \N} \Big(\frac{1}{2n+1} \sum_{j=-n}^n \|w^{(j)}\|_{B \cap f^{-j}(O)}\Big) = \infty.
    \]
\end{itemize}
\end{corollary}

\begin{proof}
We shall present the proof of (b) (the proof of (a) is simpler and is left to the reader).
Since the necessity of the condition follows trivially from Theorem~\ref{ExpCX-T1}, let us prove its sufficiency.
Take an arbitrary $O \in \cO$. Since $\bigcup_{n \in \Z} f^n(W)$ is dense in $X$, there exists $k \in \Z$ such that
$O \cap f^k(W) \neq \varnothing$.
Hence, there exists $O' \in \cO$ such that $O' \subset f^{-k}(O) \cap W$.
Since $O' \subset W$, the condition in (b) gives $B \in \cB$ such that
\[
\sup_{n \in \N} \Big(\frac{1}{2n+1} \sum_{j=-n}^n \|w^{(j)}\|_{B \cap f^{-j}(O')}\Big) = \infty,
\]
which implies that
\begin{equation}\label{ExpCX-Eq11}
\sup_{n \in \N} \Big(\frac{1}{2n+1} \sum_{j=-n}^n \|w^{(j)}\|_{B \cap f^{-(j+k)}(O)}\Big) = \infty.
\end{equation}
By hypothesis, there exists a constant $C > 1$ such that $C^{-1} \leq |w(x)| \leq C$ for all $x \in X$.
A simple computation shows that
\[
\|w^{(j+m)}\|_A \leq C^{|m|} \|w^{(j)}\|_A \ \ \text{ for all } A \in \cB \text{ and } j,m \in \Z.
\]
Thus,
\[
\sum_{j=-n}^n \|w^{(j)}\|_{B \cap f^{-(j+k)}(O)} \leq C^{|k|} \sum_{j=-n}^{n} \|w^{(j+k)}\|_{B \cap f^{-(j+k)}(O)}
\leq C^{|k|} \sum_{j=-(n+|k|)}^{n+|k|} \|w^{(j)}\|_{B \cap f^{-j}(O)}.
\]
In view of (\ref{ExpCX-Eq11}), we conclude that 
\[
\sup_{n \in \N} \Big(\frac{1}{2n+1} \sum_{j=-n}^n \|w^{(j)}\|_{B \cap f^{-j}(O)}\Big) = \infty.
\]
Hence, by Theorem~\ref{ExpCX-T1}, $C_{w,f}$ is average expansive.
\end{proof}

\begin{example}\label{ExpCX-Ex5}
Consider a {\em bilateral weighted backward shift}
\[
B_w : (x_n)_{n \in \Z} \in c_0(\Z) \mapsto (w_n x_{n+1})_{n \in \Z} \in c_0(\Z),
\]
where $w = (w_n)_{n \in \Z}$ is a bounded sequence of scalars with $\inf_{n \in \Z} |w_n| > 0$.
Note that $B_w$ can be regarded as the weighted composition operator $C_{w,f}$ on $C_0(X)$ 
corresponding to $X = \Z$ with the discrete topology and $f : n \in \Z \mapsto n+1 \in \Z$.
In this case, the set $W = \{0\}$ satisfies the condition required in Corollary~\ref{ExpCX-C1}.
Since
\[
\|w^{(n)}\|_{f^{-n}(\{0\})} = |w_{-n} \cdots w_{-1}| \ \ \text{ and } \ \ 
\|w^{(-n)}\|_{f^{n}(\{0\})} = |w_0 \cdots w_{n-1}|^{-1} \ \ \ (n \geq 1),
\]
we conclude that:
\begin{itemize}
\item $B_w$ is expansive if and only if
  \[
  \sup_{n \in \N} |w_{-n} \cdots w_{-1}| = \infty \ \ \text{ or } \ \ \sup_{n \in \N}  |w_1 \cdots w_n|^{-1} = \infty.
  \]
\item $B_w$ is average expansive if and only if
  \[
  \sup_{n \in \N} \Big(\frac{1}{n} \sum_{j=1}^n |w_{-j} \cdots w_{-1}|\Big) = \infty \ \ \text{ or } \ \ 
  \sup_{n \in \N} \Big(\frac{1}{n} \sum_{j=1}^n \frac{1}{|w_1 \cdots w_j|}\Big) = \infty.
  \]
\end{itemize}
In this way, we recovered the characterizations of expansivity and average expansivity obtained 
in~\cite{BerCirDarMesPuj18} and~\cite{AlvBerMes21}, respectively.
The characterization of uniform expansivity was also obtained in~\cite{BerCirDarMesPuj18}
and can be recovered from Example~\ref{ExpCX-Ex1}(c). 
In fact, we have that 
\begin{itemize}
\item $B_w$ is uniformly expansive if and only if there is a partition $(I,J)$ of $\Z$ such that
  \[
  \lim_{n \to \infty} \inf_{k \in I} |w_{k-n+1} \cdots w_k| = \infty \ \ \text{ and } \ \
  \lim_{n \to \infty} \inf_{k \in J} |w_{k+1} \cdots w_{k+n}|^{-1} = \infty,
  \]
  where the infimum over the empty set is defined as $\infty$.
\end{itemize}
With the help of \cite[Lemma~31]{BerCirDarMesPuj18}, we have that the existence of such a partition $(I,J)$ is equivalent to
the fact that one of the partitions $(\varnothing,\Z)$, $(\Z,\varnothing)$ or $(-\N,\N_0)$ does the job,
which corresponds to the characterization given in \cite{BerCirDarMesPuj18}.
\end{example}

\begin{remark}\label{ExpCX-R1}
For notions of expansivity that do not require the operator to be invertible, 
we need to assume that $f$, $w$ and $M$ satisfy only the following conditions:
\begin{itemize}
\item $f : X \to X$ is a continuous function satisfying $f(\fS) \subset \fS$,
\item $w : X \to \K$ is a continuous function which is bounded on each set $B \in \fS$,
\item $C_{w,f}(M) \subset M$ and (C2) holds.
\end{itemize}
Under these hypotheses, by choosing any base $\cB$ of the bornology $\fS$ and any base $\cO$ of the topology of $X$
with $\varnothing \not\in \cB$ and $\varnothing \not\in \cO$, we have that 
the following characterizations hold for the weighted composition operator $C_{w,f}$ on $M$:

\begin{itemize}
\item [\rm (a)] $C_{w,f}$ is positively expansive if and only if, for each $O \in \cO$, there exists $B \in \cB$ such that
    \[
    \sup_{n \in \N} \|w^{(n)}\|_{B \cap f^{-n}(O)} = \infty.
    \]
\item [\rm (b)] $C_{w,f}$ is average positively expansive if and only if, for each $O \in \cO$, there exists $B \in \cB$ such that
    \[
    \sup_{n \in \N} \Big(\frac{1}{n} \sum_{j=0}^{n-1} \|w^{(j)}\|_{B \cap f^{-j}(O)}\Big) = \infty.
    \]
\item [\rm (c)] $C_{w,f}$ is uniformly positively expansive if and only if, for each $A \in \cB$, there exists $B \in \cB$ such that 
\[
\lim_{n \to \infty} \inf_{O \in \cO_A} \|w^{(n)}\|_{B \cap f^{-n}(O)} = \infty.
\]
\end{itemize}
Examples analogous to the ones given above can also be given in the non-invertible case, but we omit the details. 
Moreover, if $X$ is a locally compact Hausdorff space, then we can choose $\cO$ to be a base of relatively compact open sets in $X$.
\end{remark}


\section{Expansive weighted composition operators on $L^p(\mu)$}\label{S-ExpLp}

Throughout this section, we fix a real number $p \in [1,\infty)$ and an arbitrary positive measure space $(X,\fM,\mu)$,
unless otherwise specified. We denote by $\fF$ the set of all measurable sets $B \in \fM$ satisfying $0 < \mu(B) < \infty$.
$L^p(\mu)$ denotes the Banach space over $\K$ of all $\K$-valued $p$-integrable functions on $(X,\fM,\mu)$ 
endowed with the $p$-norm
\[
\|\varphi\|_p = \Big(\int_X |\varphi|^p d\mu\Big)^{\frac{1}{p}}.
\]
We fix a {\em weight function} $w : X \to \K$, that is, a measurable map satisfying
\[ 
\varphi \cdot w \in L^p(\mu) \ \text{ for all } \varphi \in L^p(\mu).
\]
We also fix a bimeasurable map $f : X \to X$ (i.e., $f(B) \in \fM$ and $f^{-1}(B) \in \fM$ whenever $B \in \fM$)
for which the {\em weighted composition operator}
\[
C_{w,f}(\varphi) = w \cdot (\varphi \circ f)
\]
is a well-defined bounded linear operator on $L^p(\mu)$. 
It is not difficult to show that this is the case if and only if there is a constant $c > 0$ such that
\[ 
\int_B |w|^p d\mu \leq c\, \mu(f(B)) \ \text{ for every } B \in \fM.
\]

Since we are primarily interested in invertible operators, we will also assume that $f$ is bijective and that $\wt{w} = \frac{1}{w \circ f^{-1}}$
is a well-defined measurable map on $X$ satisfying
\[ 
\varphi \cdot \wt{w} \in L^p(\mu) \ \text{ for all } \varphi \in L^p(\mu)
\]
and
\[
\int_B |\wt{w}|^p d\mu \leq \tilde{c}\, \mu(f^{-1}(B)) \ \text{ for every } B \in \fM,
\]
where $\tilde{c} > 0$ is a constant. In this way, $C_{w,f}$ is an invertible operator on $L^p(\mu)$ with 
\[
(C_{w,f})^{-1} = C_{\wt{w},f^{-1}}.
\]

We associate to $p$, $w$ and $f$ the following bilateral sequence of positive measures on $(X,\fM)$:
\[
\mu_n(B) = \int_B |w^{(n)}|^p d\mu \ \ \ \ \ (B \in \fM, n \in \Z),
\]
where $w^{(0)} = 1$, $w^{(1)} = w$, $w^{(-1)} = \wt{w}$,
\[
w^{(n)} = (w \circ f^{n-1}) \cdots (w \circ f) \cdot w \ \text{ and } \ 
w^{(-n)} = (\wt{w} \circ f^{-(n-1)}) \cdots (\wt{w} \circ f^{-1}) \cdot \wt{w} \ \ \ (n \geq 2).
\]

We now establish the main result of this section, namely,
characterizations of the various notions of expansivity for invertible weighted composition operators on $L^p(\mu)$.

\begin{theorem}\label{ExpLp-T1}
For the invertible weighted composition operator $C_{w,f}$ on $L^p(\mu)$, the following statements hold:
\begin{itemize}
\item [\rm (a)] $C_{w,f}$ is expansive if and only if 
    \[ 
    \sup_{n \in \Z} \mu_n(f^{-n}(B)) = \infty \ \text{ for all } B \in \fF.
    \]
\item [\rm (b)] $C_{w,f}$ is average expansive if and only if
   \[
    \sup_{n \in \N} \Big(\frac{1}{2n+1} \sum_{j=-n}^n \mu_j(f^{-j}(B))^\frac{1}{p}\Big) = \infty \ \text{ for all } B \in \fF.
    \]
\item [\rm (c)] $C_{w,f}$ is uniformly expansive if and only if $\fF = \fF^{+} \cup \fF^{-}$, where
    \[
    \lim_{n \to \infty} \inf_{B \in \fF^+} \frac{\mu_n(f^{-n}(B))}{\mu(B)} = \infty \ \ \text{ and } \ \
    \lim_{n \to \infty} \inf_{B \in \fF^-} \frac{\mu_{-n}(f^n(B))}{\mu(B)} = \infty.
    \]
\end{itemize}
\end{theorem}

We mention that the proof of item (c) presented below is an adaptation of the original proof given by Maiuriello \cite{MMai22} 
in the unweighted case.

\begin{proof}
(c): Suppose that $\fF = \fF^{+} \cup \fF^{-}$ as in (c).
Let $\cS$ be the set of all simple functions that belong to $S_{L^p(\mu)}$. To each
\[
\psi = \sum_{i=1}^r \alpha_i\, \rchi_{B_i} \in \cS,
\]
where $r \in \N$, $\alpha_1,\ldots,\alpha_r \in \K$ are pairwise distinct and $B_1,\ldots,B_r \in \fF$ are pairwise disjoint, 
we associate the simple functions
\[
\psi^+ = \sum_{1 \leq i \leq r, B_i \in \fF^+} \alpha_i\, \rchi_{B_i} \ \ \text{ and } \ \ 
\psi^- = \sum_{1 \leq i \leq r, B_i \not\in \fF^+} \alpha_i\, \rchi_{B_i}.
\]
Let
\[
\cA^\pm = \{\psi \in \cS : \|\psi^\pm\|_p^p \geq 1/2\}.
\]
Clearly, $\cS = \cA^+ \cup \cA^-$. For each $\varphi \in S_{L^p(\mu)}$, we can choose a sequence 
$(\varphi_k)_{k \in \N}$ in $\cS$ such that $\varphi_k \to \varphi$ in $L^p(\mu)$. Let
\[
S^\pm = \{\varphi \in S_{L^p(\mu)} : \varphi_k \in \cA^\pm \text{ for infinitely many values of } k\}.
\]
Then,
\[
S_{L^p(\mu)} = S^+ \cup S^-.
\]

Now, fix any constant $C > 0$. Since the limits in (c) are infinite, there exists $n_0 \in \N$ such that
\[
\frac{\mu_n(f^{-n}(A))}{\mu(A)} \geq C \ \text{ and } \ \frac{\mu_{-n}(f^n(B))}{\mu(B)} \geq C 
  \ \text{ whenever } A \in \fF^+, B \in \fF^- \text{ and } n \geq n_0.
\]
We claim that
\begin{equation}\label{ExpLp-Eq9}
\|(C_{w,f})^n(\psi)\|_p^p \geq \frac{C}{2} \ \ \text{ whenever } \psi \in \cA^+ \text{ and } n \geq n_0.
\end{equation}
Indeed, for any $\psi = \sum_{i=1}^r \alpha_i\, \rchi_{B_i} \in \cA^+$ (same notation as above) and any $n \geq n_0$, 
\begin{align*}
\|(C_{w,f})^n(\psi)\|_p^p 
  &= \int_X \Big|\sum_{i=1}^r \alpha_i\, \rchi_{f^{-n}(B_i)}\Big|^p |w^{(n)}|^p d\mu
  = \sum_{i=1}^r |\alpha_i|^p \mu_n(f^{-n}(B_i))\\
  &\geq \sum_{1 \leq i \leq r, B_i \in \fF^+} |\alpha_i|^p \mu_n(f^{-n}(B_i))
  \geq C \sum_{1 \leq i \leq r, B_i \in \fF^+} |\alpha_i|^p \mu(B_i)\\
  &= C\, \|\psi^+\|_p^p \geq \frac{C}{2}\cdot
\end{align*}
Similarly,
\begin{equation}\label{ExpLp-Eq10}
\|(C_{w,f})^{-n}(\psi)\|_p^p \geq \frac{C}{2} \ \ \text{ whenever } \psi \in \cA^- \text{ and } n \geq n_0.
\end{equation}
If $\varphi \in S^+$ and $n \geq n_0$, it follows from (\ref{ExpLp-Eq9}) that
$\|(C_{w,f})^n(\varphi_k)\|_p^p \geq C/2$ for infinitely many values of $k$.
Hence, by letting $k \to \infty$ along this subsequence, we see that
\[
\|(C_{w,f})^n(\varphi)\|_p^p \geq \frac{C}{2} \ \ \text{ whenever } \varphi \in S^+ \text{ and } n \geq n_0.
\]
Similarly,  it follows from (\ref{ExpLp-Eq10}) that
\[
\|(C_{w,f})^{-n}(\varphi)\|_p^p \geq \frac{C}{2} \ \ \text{ whenever } \varphi \in S^- \text{ and } n \geq n_0.
\]
This proves that 
\[
\|(C_{w,f})^n(\varphi)\|_p \to \infty \text{ uniformly on } S^+ \ \text{ and } \ 
\|(C_{w,f})^{-n}(\varphi)\|_p \to \infty \text{ uniformly on } S^-.
\]
Therefore, $C_{w,f}$ is uniformly expansive.

Conversely, suppose that $C_{w,f}$ is uniformly expansive. Then, $S_{L^p(\mu)} = S^+ \cup S^-$, where 
\begin{equation}\label{ExpLp-Eq8}
\lim_{n \to \infty} \inf_{\varphi \in S^+} \|(C_{w,f})^n(\varphi)\|_p = \infty \ \ \text{ and } \ \ 
\lim_{n \to \infty} \inf_{\varphi \in S^-} \|(C_{w,f})^{-n}(\varphi)\|_p = \infty.
\end{equation}
Let
\[
\fF^+ = \Big\{B \in \fF : \frac{\rchi_B}{\mu(B)^{1/p}} \in S^+\Big\} \ \ \text{ and } \ \ 
\fF^- = \Big\{B \in \fF : \frac{\rchi_B}{\mu(B)^{1/p}} \in S^-\Big\}.
\]
Clearly, $\fF = \fF^+ \cup \fF^-$. Moreover, (\ref{ExpLp-Eq8}) implies that the limits in (c) are infinite.

\smallskip\noindent
(b): Suppose that the supremum in (b) is infinite for all $B \in \fF$. Given $\varphi \in L^p(\mu) \backslash \{0\}$, 
there exists $\eps >0$ such that $A = \{x \in X : |\varphi(x)| \geq \eps\} \in \fF$. Since 
\[
\|(C_{w,f})^n(\varphi)\|_p^p \geq \int_{f^{-n}(A)} |\varphi \circ f^n|^p |w^{(n)}|^p d\mu \geq \eps^p \mu_n(f^{-n}(A))
\ \ \ (n \in \Z),
\]
we obtain
\[
\sup_{n \in \N} \Big(\frac{1}{2n+1} \sum_{j=-n}^n \|(C_{w,f})^j(\varphi)\|_p\Big) \geq
\eps\, \sup_{n \in \N} \Big(\frac{1}{2n+1} \sum_{j=-n}^n \mu_j(f^{-j}(A))^\frac{1}{p}\Big) = \infty,
\]
proving that $C_{w,f}$ is average expansive.

Conversely, suppose that $C_{w,f}$ is average expansive, that is,
\[
\sup_{n \in \N} \Big(\frac{1}{2n+1} \sum_{j=-n}^n \|(C_{w,f})^j(\varphi)\|_p\Big) = \infty 
  \ \text{ for all } \varphi \in L^p(\mu) \backslash \{0\}.
\]
In particular, for each $B \in \fF$, since
\[
\|(C_{w,f})^n(\rchi_B)\|_p^p = \int_X |\rchi_B \circ f^n|^p |w^{(n)}|^p d\mu = \int_{f^{-n}(B)} |w^{(n)}|^p d\mu = \mu_n(f^{-n}(B))
\ \ \ (n \in \Z),
\]
we obtain
\[
\sup_{n \in \N} \Big(\frac{1}{2n+1} \sum_{j=-n}^n \mu_j(f^{-j}(B))^\frac{1}{p}\Big) =
\sup_{n \in \N} \Big(\frac{1}{2n+1} \sum_{j=-n}^n \|(C_{w,f})^j(\rchi_B)\|_p\Big) = \infty.
\]

\noindent
(a): The proof is similar (but slightly simpler) to the proof of (b) and we omit it.
\end{proof}

\begin{example} \label{ExpLp-Ex1}
If $X = \Z$, $\fM = \cP(\Z)$ (the power set of $\Z$), $\mu$ is the counting measure on $\fM$ and $f : n \in \Z \mapsto n+1 \in \Z$, 
then $L^p(\mu) = \ell^p(\Z)$ and $C_{w,f}$ coincides with the {\em bilateral weighted backward shift}
\[
B_w : (x_n)_{n \in \Z} \in \ell^p(\Z) \mapsto (w_nx_{n+1})_{n \in \Z} \in \ell^p(\Z).
\]
In this case, the general characterizations given in Theorem~\ref{ExpLp-T1} can be used to recover the known characterizations
of expansivity, average expansivity and uniform expansivity for $B_w$ obtained in~\cite{AlvBerMes21} and~\cite{BerCirDarMesPuj18},
but we leave the details to the reader.
\end{example}

\begin{example}
If $X = \R$, $\fM$ is the $\sigma$-algebra of all Lebesgue measurable sets in $\R$, $\mu$ is the Lebesgue measure on $\fM$ 
and $f : x \in \R \mapsto x+1 \in \R$, then $L^p(\mu) = L^p(\R)$ and $C_{w,f}$ coincides with the 
{\em bilateral weighted translation operator}
\[
T_w : \varphi \in L^p(\R) \mapsto w(\cdot) \varphi(\cdot + 1) \in L^p(\R).
\]
Theorem~\ref{ExpLp-T1} can be immediately applied to provide characterizations of expansivity, average expansivity 
and uniform expansivity for $T_w$.
\end{example}

\begin{remark}
For notions of expansivity that do not require the operator to be invertible, 
we need not assume that $w$ and $f$ satisfy the additional hypotheses considered in the second paragraph of this section. 
In this case, the following characterizations hold for the weighted composition operator $C_{w,f}$ on $L^p(\mu)$:
\begin{itemize}
\item [\rm (a)] $C_{w,f}$ is positively expansive if and only if 
    \[ 
    \sup_{n \in \N} \mu_n(f^{-n}(B)) = \infty \ \text{ for all } B \in \fF.
    \]
\item [\rm (b)] $C_{w,f}$ is average positively expansive if and only if
   \[ 
    \sup_{n \in \N} \Big(\frac{1}{n} \sum_{j=0}^{n-1} \mu_j(f^{-j}(B))^\frac{1}{p}\Big) = \infty \ \text{ for all } B \in \fF.
    \]
\item [\rm (c)] $C_{w,f}$ is uniformly positively expansive if and only if 
    \[ 
    \lim_{n \to \infty} \inf_{B \in \fF} \frac{\mu_n(f^{-n}(B))}{\mu(B)} = \infty.
    \]
\end{itemize}
These characterizations can be applied to unilateral weighted backward shifts $B_w$ on $\ell^p(\N)$
(recovering the characterizations obtained in~\cite{AlvBerMes21} and~\cite{BerCirDarMesPuj18}) and
to unilateral weighted translation operators $T_w$ on $L^p[1,\infty)$, for instance.
\end{remark}


\section{Dissipative weighted composition operators on $L^p(\mu)$}\label{S-Dissipative}

Throughout this section, we fix a real number $p \in [1,\infty)$.

\begin{definition}\label{Dis-D1}
A {\em weighted $p$-measurable system} is a $5$-tuple $(X,\fM,\mu,f,w)$ such that the following conditions hold:
\begin{itemize}
\item [(a)] $(X,\fM,\mu)$ is a $\sigma$-finite measure space with $\mu(X) > 0$;
\item [(b)] $f : X \to X$ is a non-singular bimeasurable bijective map, 
  where $f$ {\em non-singular} means that $\mu(f^{-1}(B)) = 0$ if and only if $\mu(B) = 0$;
\item [(c)] $w : X \to \K$ is a bounded measurable map for which there is a constant $c > 0$ such that
  \[
  \int_B |w|^p d\mu \leq c\, \mu(f(B)) \ \text{ for all } B \in \fM.
  \]
\end{itemize}
\end{definition}

Note that condition (c) implies that the weighted composition operator 
\[
C_{w,f}(\varphi) = w \cdot (\varphi \circ f)
\]
is a well-defined bounded linear operator on $L^p(\mu)$.

In the particular case in which $w$ is the constant function $1$, 
we have the known concept of a {\em measurable system} $(X,\fM,\mu,f)$.
Recall that such a system (or the map $f$) is said to be:
\begin{itemize}
\item {\em conservative} if for each measurable set $B$ of positive $\mu$-measure, there exists $n \in \N$ such that 
  $\mu(B \cap f^{-n}(B)) > 0$;
\item {\em dissipative} if there is a measurable set $W$ (called a {\em wandering set}) such that the sets $f^n(W)$, $n \in \Z$, 
  are pairwise disjoint and $X = \bigcup_{n \in \Z} f^n(W)$. 
\end{itemize}
Recall also that a set $A \subset  X$ is said to be $f$-invariant if $ f^{-1}(A) = A$. 
The following classic result can be found in \cite[Theorem 3.2]{UKre85}:

\medskip\noindent
{\bf Hopf Decomposition Theorem.}  {\it If $(X,\fM,\mu)$ is a $\sigma$-finite measure space and $f : X \to X$ is a non-singular
measurable map, then $X$ can be written as the union of two disjoint $f$-invariant sets $\cC(f)$ and $\cD(f)$, 
called the {\em conservative} and the {\em dissipative parts} of $f$, respectively, such that 
$f|_{\cC(f)}$ is conservative and $f|_{\cD(f)}$ is dissipative.}

\medskip
As mentioned in the Introduction, the dynamics of composition operators $C_f$ on $L^p(\mu)$ associated to
dissipative measurable systems has been investigated in a series of recent papers.

\begin{definition}\label{Dis-D2}
We say that a weighted $p$-measurable system $(X,\fM,\mu,f,w)$ is {\em invertible} if 
$\wt{w} = \frac{1}{w \circ f^{-1}}$ is a well-defined bounded measurable map on $X$ such that there is a constant $\tilde{c} > 0$ satisfying
\[
\int_B |\wt{w}|^p d\mu \leq \tilde{c}\, \mu(f^{-1}(B)) \ \text{ for every } B \in \fM.
\]
\end{definition}

In this case, $C_{w,f}$ is an invertible operator on $L^p(\mu)$ with 
\[
(C_{w,f})^{-1} = C_{\wt{w},f^{-1}}.
\]

For the remainder of this section, we fix a dissipative invertible weighted $p$-measurable system $(X,\fM,\mu,f,w)$
with wandering set $W$ satisfying $0 < \mu(W) < \infty$, unless otherwise specified.
Moreover, the functions $w^{(n)}$ and the measures $\mu_n$ are defined exactly as in the previous section.

\begin{theorem}\label{Dis-T1}
If $\nu$ is the measure on $(X,\fM)$ defined by 
\begin{equation}\label{Dis-Eq1}
\nu(B) = \sum_{n \in \Z} \mu_n(B \cap f^{-n}(W)),
\end{equation}
then the weighted composition operator $C_{w,f}$ on $L^p(\mu)$ is conjugate to the composition operator $C_f$ on $L^p(\nu)$
by means of the isometric isomorphism $\Pi : L^p(\mu) \to L^p(\nu)$ given by
\begin{equation}\label{Dis-Eq2}
\Pi(\varphi) =\sum_{n \in \Z}\, \frac{\varphi}{w^{(n)}}\, \rchi_{f^{-n}(W)}\cdot
\end{equation}
\end{theorem}

\begin{figure}[h]
\centering
\begin{tikzpicture}[scale=0.25,>=stealth]

\node[right] at (25,21) {$C_{w,f}$};
  \node[right] at (10,20) {$L^p(\mu)$};
  \node[right] at (40,20) {$L^p(\mu)$};
  \node[right] at (25,11) {$C_{f}$};
\node[right] at (10,10) {$L^p(\nu)$};
  \node[right] at (40,10) {$L^p(\nu)$};

  \draw[double, ->] (15,20) -- (40,20);
  \draw[double, ->] (15,10) -- (40,10);

  \node[right] at (9,15) {$\Pi$};
  \draw[double, ->] (12,18) -- (12,12);
    \node[right] at (39,15) {$\Pi$};
   \draw[double, ->] (42,18) -- (42,12);

  \end{tikzpicture}
\caption{Dissipative composition dynamical system }
\end{figure}

\begin{proof}
Clearly, $\nu$ is a measure on $(X,\fM)$. Moreover,
\begin{equation}\label{Dis-Eq3}
\int_X g\, d\nu = \sum_{n \in \Z} \int_{f^{-n}(W)} g\, |w^{(n)}|^p\, d\mu \ \text{ for every } g \in L^1(\nu).
\end{equation}
Indeed, by linearity, (\ref{Dis-Eq1}) implies (\ref{Dis-Eq3}) for simple functions. 
By using classical approximation arguments, we obtain (\ref{Dis-Eq3}) in general.
It is also clear that $\Pi$ is linear.

Let us prove that $\Pi$ is an isometry. Indeed, for any $\varphi \in L^p(\mu)$,
\begin{align*}
\|\Pi(\varphi)\|_{p}^{p} 
  &=\int_X \Big|\sum_{n \in \Z} \frac{\varphi}{w^{(n)}}\, \rchi_{f^{-n}(W)}\Big|^p\, d\nu \\
  &= \sum_{k \in \Z} \int_{f^{-k}(W)} \Big|\sum_{n \in \Z} \frac{\varphi}{w^{(n)}}\, \rchi_{f^{-n}(W)}\Big|^p |w^{(k)}|^p\, d\mu \\
  &= \sum_{k \in \Z} \int_{f^{-k}(W)} \Big|\frac{\varphi}{w^{(k)}}\Big|^p |w^{(k)}|^p\, d\mu \\
  &= \sum_{k \in \Z} \int_{f^{-k}(W)} |\varphi|^p\, d\mu = \|\varphi\|_{p}^{p}.
\end{align*}

Let us now show that $\Pi$ is surjective. Indeed, given $\psi \in L^p(\nu)$, define 
\[
\varphi =\sum_{n \in \Z} \psi\, w^{(n)} \rchi_{f^{-n}(W)}.
\]
Note that $\varphi \in L^p(\mu)$, since
\begin{align*}
\int_X |\varphi|^p\, d\mu
  &= \int_X \Big|\sum_{n \in \Z} \psi\, w^{(n)} \rchi_{f^{-n}(W)} \Big|^p\, d\mu \\
  &=  \int_X \sum_{n \in \Z} |\psi\, w^{(n)} \rchi_{f^{-n}(W)}|^p\, d\mu \\
  &=  \sum_{n \in \Z} \int_{f^{-n}(W)} |\psi|^p\, |w^{(n)}|^p\, d\mu \\
  &= \int_X |\psi|^p\, d\nu < \infty. 
\end{align*}
Moreover,
\begin{align*} 
\Pi(\varphi)
  &= \sum_{n \in \Z} \frac{1}{w^{(n)}} \Big(\sum_{m \in \Z} \psi\, w^{(m)} \rchi_{f^{-m}(W)}\Big) \rchi_{f^{-n}(W)} \\
  &=\sum_{n \in \Z} \sum_{m \in \Z} \frac{\psi\, w^{(m)}}{w^{(n)}}\, \rchi_{f^{-m}(W)} \rchi_{f^{-n}(W)} \\
  &= \sum_{n \in \Z} \psi\, \rchi_{f^{-n}(W)} = \psi.
\end{align*}

Thus, $\Pi$ is an isometric isomorphism from $L^p(\mu)$ onto $L^p(\nu)$. 

Finally, let us prove that $\Pi \circ C_{w,f} = C_f \circ \Pi$. Take $\varphi \in L^p(\mu)$. On the one hand, 
\[
\Pi (C_{w,f}(\varphi)) = \Pi(w \cdot (\varphi \circ f)) = \sum_{n \in \Z}  \frac{w \cdot (\varphi \circ f)}{w^{(n)}}\, \rchi_{f^{-n}(W)}\cdot
\]
On the other hand, 
\begin{align*}
C_f(\Pi(\varphi))
  &= C_f\Big(\sum_{n \in \Z} \frac{\varphi}{w^{(n)}}\, \rchi_{f^{-n}(W)}\Big)\\
  &= \sum_{n \in \Z} \frac{\varphi \circ f}{w^{(n)}\circ f}\, \rchi_{f^{-(n+1)}(W)}\\
  &= \sum_{n \in \Z} \frac{w \cdot (\varphi \circ f)}{w^{(n+1)}}\, \rchi_{f^{-(n+1)}(W)}\cdot
\end{align*}

This completes the proof.
\end{proof}

The previous theorem provides us with a method for extending results 
about dissipative unweighted composition operators to the weighted case.
We will illustrate this method with the results below. First, let us recall some definitions.
The {\em lower density} of a set $A \subset \N$ is defined by
\[
\ldens(A) = \liminf_{n \to \infty} \frac{\card(A \cap [1,n])}{n}\cdot
\]
Let $T$ be an operator on a Banach space $Y$. For any subsets $U$ and $V$ of $Y$, the {\em return set from $U$ to $V$} is given by
\[
N(U,V) = \{n \in \N_0 : T^n(U) \cap V \neq \varnothing\}.
\]
The operator $T$ is said to be:
\begin{itemize}
\item {\em topologically transitive} (resp.\ {\em topologically mixing}) if $N(U,V)$ is nonempty (resp.\ co-finite) 
  whenever $U$ and $V$ are nonempty open sets in $Y$.
\item {\em Devaney chaotic} if it is topologically transitive and has a dense set of periodic points.
\item {\em hypercyclic} (resp.\ {\em frequently hypercyclic}) if there is a vector $y \in Y$ such that $N(\{y\},V)$ is nonempty 
  (resp.\ has positive lower density) for every nonempty open set $V$ in $Y$.
\item {\em recurrent} (resp.\ {\em frequently recurrent}) if there is dense set of vectors $y \in Y$ with the property that
  $N(\{y\},V)$ is nonempty (resp.\ has positive lower density) for every neighborhood $V$ of $y$ in $Y$.
\item {\em Li-Yorke chaotic} if there is an uncountable set $S \subset Y$ such that each pair $(y,z)$ of distinct points in $S$ satisfies
  \[
  \liminf_{n \to \infty} \|T^n y - T^n z\| = 0 \ \ \text{ and } \ \ \limsup_{n \to \infty} \|T^n y - T^n z\| > 0.
  \]
\end{itemize}
Moreover, we say that an invertible operator $T$ on $Y$ has the {\em shadowing property} if, for every $\eps > 0$, 
there exists $\delta > 0$ such that for every bilateral sequence $(y_j)_{j \in \Z}$ in $Y$ satisfying
\[
\|T y_j - y_{j+1}\| \leq \delta \ \text{ for all } j \in \Z,
\]
there exists $y \in Y$ such that
\[
\|y_j - T^j y\| < \eps \ \text{ for all } j \in \Z.
\]

\smallskip
Recall that $(X,\fM,\mu,f,w)$ is a dissipative invertible weighted $p$-measurable system
with wandering set $W$ satisfying $0 < \mu(W) < \infty$.
Moreover, the measure $\nu$ is as in Theorem~\ref{Dis-T1}.

\begin{theorem}\label{Dis-T2}
$C_{w,f}$ is Devaney chaotic on $L^p(\mu)$ if and only if the following condition holds:
\begin{itemize}
\item For each $\eps > 0$ and $B \in \fM$ with $\nu(B) < \infty$, there exists a measurable set $B' \subset B$ such that 
  \[
  \nu(B \backslash B') < \eps \ \ \text{ and } \ \ \sum _{n\in \Z} \nu (f^{n}(B')) < \infty.
  \]
\end{itemize}
\end{theorem}

\begin{proof}
By Theorem~\ref{Dis-T1}, $C_{w,f}$ is Devaney chaotic on $L^p(\mu)$ if and only if $C_{f}$ is Devaney chaotic on $L^p(\nu)$.
Now, by \cite[Corollary~3.5]{UDarBPir21}, $C_{f}$ is Devaney chaotic on $L^p(\nu)$ if and only if the above condition holds.
\end{proof}

\begin{theorem}\label{Dis-T3}
If 
\[
\sum _{n \in \Z} \mu_n(f^{-n}(W)) < \infty,
\]
then $C_{w,f}$ is Devaney chaotic, topologically mixing and frequently hypercyclic on $L^p(\mu)$.
\end{theorem}

\begin{proof}
By Theorem~\ref{Dis-T1}, $C_{w,f}$ is Devaney chaotic, topologically mixing or frequently hypercyclic on $L^p(\mu) $ if and only if 
$C_{f}$ is Devaney chaotic, topologically mixing or frequently hypercyclic on $L^p(\nu)$, respectively. 
Now, since $\nu(X) = \sum_{n \in \Z} \mu_n(f^{-n}(W)) < \infty$ by hypothesis, it follows from \cite[Theorem~3.3]{UDarBPir21} that
$C_{f}$ is simultaneouly Devaney chaotic, topologically mixing and frequently hypercyclic on $L^p(\nu)$.
\end{proof}

\begin{theorem}\label{Dis-T4}
$C_{w,f}$ is hypercyclic on $L^p(\mu)$ if and only if $C_{w,f}$ is recurrent on $L^p(\mu)$.
\end{theorem}

\begin{proof}
By Theorem~\ref{Dis-T1}, $C_{w,f}$ is hypercyclic or recurrent on $L^p(\mu)$ if and only if 
$C_{f}$ is hypercyclic or recurrent on $L^p(\nu)$, respectively. On the other hand, 
by \cite[Theorem 2.2]{AniMaiSeo24}, $C_{f}$ is hypercyclic on $L^p(\nu)$ if and only if $C_{f}$ is recurrent on $L^p(\nu)$.
\end{proof}

Let us now extend the concept of {\em bounded distortion} \cite{AniDarMai21,UDarBPir21} to our context.

\begin{definition}\label{Dis-D3}
We say that the dissipative invertible weighted $p$-measurable system $(X,\fM,\mu,f,w)$ is of {\em bounded distortion} 
on the wandering set $W$ if there is a constant $K > 0$ such that
\begin{equation}\label{Dis-Eq4}
\frac{1}{K}\, \frac{\mu_{-k}(f^k(W))}{\mu(W)} \leq \frac{\mu_{-k}(f^k(B))}{\mu(B)} \leq K\, \frac{\mu_{-k}(f^k(W))}{\mu(W)}
\end{equation}
for all $k \in \Z$ and all measurable set $B \subset W$ with $\mu(B) > 0$. 
\end{definition}

\begin{lemma}\label{Dis-L1}
$(X,\fM,\mu,f,w)$ is of bounded distortion on $W$ if and only if so is $(X,\fM,\nu,f)$.
\end{lemma}

\begin{proof}
For any $k \in \Z$ and any measurable set $B \subset W$ with $\mu(B) > 0$, we have that
\[
\nu(f^k(B)) = \sum_{n \in \Z} \mu_n(f^k(B) \cap f^{-n}(W)) = \mu_{-k}(f^k(B)).
\]
Hence, (\ref{Dis-Eq4}) can be rewritten as
\begin{equation}\label{Dis-Eq5}
\frac{1}{K}\, \frac{\nu(f^k(W))}{\nu(W)} \leq \frac{\nu(f^k(B))}{\nu(B)} \leq K\, \frac{\nu(f^k(W))}{\nu(W)}\cdot
\end{equation}
Since $(X,\fM,\nu,f)$ is of bounded distortion on $W$ exactly when there is a constant $K > 0$ such that (\ref{Dis-Eq5}) holds
for all $k \in \Z$ and all measurable set $B \subset W$ with $\nu(B) > 0$, we are done.
\end{proof}

Given Banach spaces $Y$ and $Z$, and given operators $T : Y \to Y$ and $S : Z \to Z$, 
recall that $S$ is said to be a {\em linear factor} of $T$ if there is a {\em linear factor map} (i.e., a bounded linear surjection)
$\Phi : Y \to Z$ for which the following diagram commutes:

\[ 
 \xymatrix{
 Y
 \ar[rrr]^T
 \ar[d]_\Phi & & &
 Y
 \ar[d]^{\Phi}\\
 Z \ar[rrr]^{S} & & & Z }
\]

\begin{theorem}\label{Dis-T5}
If the system $(X,\fM,\mu, f, w)$ is of bounded distortion on $W$, then the weighted composition operator $C_{w,f}$ on $L^p(\mu)$ 
has the bilateral weighted backward shift $B_u$ on $\ell^p(\Z)$ as a linear factor, where the sequence $u = (u_k)_{k \in \Z}$ is defined by
\[
u_k = \Big(\frac{\mu_{-k+1}(f^{k-1}(W))}{\mu_{-k}(f^k(W))}\Big)^\frac{1}{p}
\]
and the linear factor map is $\Gamma \circ \Pi$, where $\Pi : L^p(\mu) \to L^p(\nu)$ is the isometric isomorphism 
given by (\ref{Dis-Eq2}) and $\Gamma : L^p(\nu) \to \ell^p(\Z)$ is the bounded linear surjection defined by
\[
\Gamma(\varphi)(k) = \frac{\mu_{-k}(f^k(W))^\frac{1}{p}}{\mu(W)} \int_W \varphi \circ f^k\, d\mu.
\]
Moreover, $C_{w,f}$ has Property P if and only if $B_u$ has Property P, where Property P can be any one of the following properties: 
(1) Li-Yorke chaos; (2) hypercyclicity; (3) topological mixing; (4) Devaney chaos; (5) frequent hypercyclicity; 
(6) expansivity; (7) average expansivity; (8) uniform expansivity; (9) the shadowing property.
\end{theorem}

\begin{figure}[h]
\centering
\begin{tikzpicture}[scale=0.25,>=stealth]

\node[right] at (25,21) {$C_{w,f}$};
  \node[right] at (10,20) {$L^p(\mu)$};
  \node[right] at (40,20) {$L^p(\mu)$};
   \node[right] at (25,11) {$C_{f}$};
\node[right] at (10,10) {$L^p(\nu)$};
  \node[right] at (40,10) {$L^p(\nu)$};
  \node[right] at (25,1) {$B_u$};
  \node[right] at (10,0) {$\ell^p(\Z)$};
  \node[right] at (40,0) {$\ell^p(\Z)$};
  \draw[double, ->] (15,20) -- (40,20);
  \draw[double, ->] (15,10) -- (40,10);
  \draw[double, ->] (15,0) -- (40,0);
  \draw[double, ->] (12,18) -- (12,12);
   \node[right] at (9,15) {$\Pi$};
    \node[right] at (39,15) {$\Pi$};
     \node[right] at (9,5) {$\Gamma$};
      \node[right] at (39,5) {$\Gamma$};
   \draw[double, ->] (42,18) -- (42,12);
    \draw[double, ->] (12,8) -- (12,2);
   \draw[double, ->] (42,8) -- (42, 2);
  \end{tikzpicture}
\caption{Dissipative composition dynamical system of bounded distortion }
\end{figure}

\begin{proof}
By Theorem~\ref{Dis-T1}, $\Pi : L^p(\mu) \to L^p(\nu)$ is an isometric isomorphism that 
establishes a conjugation between $C_{w,f}$ and $C_f$, that is,
\[
\Pi \circ C_{w,f} = C_f \circ \Pi.
\]
Since
\[
u_k = \Big(\frac{\mu_{-k+1}(f^{k-1}(W))}{\mu_{-k}(f^k(W))}\Big)^\frac{1}{p}  = \Big(\frac{\nu(f^{k-1}(W))}{\nu(f^k(W))}\Big)^\frac{1}{p}
\]
and 
\[
\Gamma(\varphi)(k) = \frac{\mu_{-k}(f^k(W))^\frac{1}{p}}{\mu(W)} \int_W \varphi \circ f^k\, d\mu
  = \frac{\nu(f^k(W))^\frac{1}{p}}{\nu(W)} \int_W \varphi \circ f^k\, d\nu,
\]
and since $(X,\fM,\nu,f)$ is of bounded distortion on $W$ (Lemma~\ref{Dis-L1}), it follows from \cite[Lemma~4.2.3]{AniDarMai21} 
(see also \cite{AniDarMai22}) that $\Gamma$ is a bounded linear surjection from $L^p(\nu)$ onto $\ell^p(\Z)$ such that
\[
\Gamma \circ C_f = B_u \circ \Gamma.
\]
In other words, $B_u$ is a linear factor of $C_f$ by means of the linear factor map $\Gamma$.
Thus, $\Gamma \circ \Pi$ is a bounded linear surjection from $L^p(\mu)$ onto $\ell^p(\Z)$ such that
\[
\Gamma \circ \Pi \circ C_{w,f} = B_u \circ \Gamma \circ \Pi,
\]
proving the first assertion of the theorem.

Now, in view of Theorem~\ref{Dis-T1}, 
\[
C_{w,f} \text{ has Property P on } L^p(\mu) \ \ \Longleftrightarrow \ \ C_f \text{ has Property P on } L^p(\nu).
\]
On the other hand, except for the property of average expansivity, \cite[Theorem~M]{AniDarMai22} asserts that
\[
C_f \text{ has Property P on } L^p(\nu) \ \ \Longleftrightarrow \ \ B_u \text{ has Property P on } \ell^p(\Z).
\]
Hence, to complete the proof of the last assertion of the theorem, it remains to prove that
\[
C_f \text{ is average expansive on } L^p(\nu) \ \ \Longleftrightarrow \ \ B_u \text{ is average expansive on } \ell^p(\Z).
\]

Suppose that $C_f$ is average expansive on $L^p(\nu)$.
It was shown in \cite[Lemma~3.4]{AniDarMai22} that the linear factor map $\Gamma : L^p(\nu) \to \ell^p(\Z)$ 
admits a {\em strong bounded selector}, which means that there exists a constant $L \geq 1$ such that,
for each $y \in \ell^p(\Z)$, there exists $\varphi_y \in \Gamma^{-1}(\{y\}) \subset L^p(\nu)$ with
\[
\|(B_u)^n(y)\|_p \leq L\, \|(C_f)^n(\varphi_y)\|_p \ \text{ and } \ \|(C_f)^n(\varphi_y)\|_p \leq L\, \|(B_u)^n(y)\|_p \ \ \text{ for all } n \in \Z.
\]
If $y \neq 0$ then $\varphi_y \neq 0$ and therefore
\[
\sup_{n \in \N} \Big(\frac{1}{2n+1} \sum_{j=-n}^n \|(B_u)^j(y)\|_p\Big) 
  \geq \frac{1}{L}\, \sup_{n \in \N} \Big(\frac{1}{2n+1} \sum_{j=-n}^n \|(C_f)^j(\varphi_y)\|_p\Big) = \infty,
\]
since $C_f$ is average expansive on $L^p(\nu)$. This proves that $B_u$ is average expansive on $\ell^p(\Z)$.

Conversely, suppose that $B_u$ is average expansive on $\ell^p(\Z)$.
By \cite[Proposition~15]{AlvBerMes21},
\[
\sup_{n \in \N} \Big(\frac{1}{n} \sum_{j=1}^n u_{-j+1} \cdots u_0\Big) = \infty \ \ \text{ or } \ \ 
\sup_{n \in \N} \Big(\frac{1}{n} \sum_{j=1}^n \frac{1}{u_1 \cdots u_j}\Big) = \infty.
\]
Since
\[
u_{-j+1} \cdots u_0 = \Big(\frac{\nu(f^{-j}(W))}{\nu(W)}\Big)^\frac{1}{p} \ \ \text {and } \ \ 
u_1 \cdots u_j = \Big(\frac{\nu(W)}{\nu(f^j(W))}\Big)^\frac{1}{p},
\]
we obtain
\[
\sup_{n \in \N} \Big(\frac{1}{n} \sum_{j=1}^n \nu(f^{-j}(W))^\frac{1}{p}\Big) = \infty \ \ \text{ or} \ \ 
\sup_{n \in \N} \Big(\frac{1}{n} \sum_{j=1}^n \nu(f^j(W))^\frac{1}{p}\Big) = \infty,
\]
which is equivalent to
\[
\sup_{n \in \N} \Big(\frac{1}{2n+1} \sum_{j=-n}^n \nu(f^{-j}(W))^\frac{1}{p}\Big) = \infty.
\]
Now, take $B \in \fM$ with $0 < \nu(B) < \infty$. 
There exists $k \in \Z$ such that $\nu(B \cap f^k(W)) > 0$.
Define $A = f^{-k}(B \cap f^k(W)) \subset W$. 
Since $(X,\fM,\nu,f)$ is of bounded distortion on $W$ (Lemma~\ref{Dis-L1}), there is a constant $K > 0$ such that
\[
\frac{\nu(f^j(A))}{\nu(A)} \geq \frac{1}{K} \frac{\nu(f^j(W))}{\nu(W)} \ \ \text{ for all } j \in \Z.
\]
Thus, 
\begin{align*}
\sup_{n \in \N} \Big(\frac{1}{2n+1} \sum_{j=-n}^n \nu(f^{-j}(B))^\frac{1}{p}\Big) 
  &\geq \sup_{n \in \N} \Big(\frac{1}{2n+1} \sum_{j=-n}^n \nu(f^{-j}(B\cap f^k(W)))^\frac{1}{p}\Big)\\
  &= \sup_{n \in \N} \Big(\frac{1}{2n+1} \sum_{j=-n}^n \nu(f^{-j+k}(A))^\frac{1}{p}\Big)\\
  &\geq \Big(\frac{\nu(A)}{K\,\nu(W)}\Big)^\frac{1}{p} \sup_{n \in \N} \Big(\frac{1}{2n+1} \sum_{j=-n}^n \nu(f^{-j+k}(W))^\frac{1}{p}\Big)
  = \infty.
\end{align*}
Therefore, by Theorem~\ref{ExpLp-T1}(b), $C_f$ is average expansive on $L^p(\nu)$.
\end{proof}

\begin{theorem}
If the system $(X,\fM,\mu,f,w)$ is of bounded distortion, then the following assertions are equivalent:
\begin{itemize}
\item [\rm (i)] $C_{w,f}$ is frequently  hypercyclic on $L^p(\mu)$;
\item [\rm (ii)] $C_{w,f}$ is frequently recurrent on $L^p(\mu)$;
\item [\rm (iii)] $C_{w,f}$ is Devaney chaotic on $L^p(\mu)$.
\end{itemize}

\end{theorem}
\begin{proof}

By Theorem~\ref{Dis-T1}, $C_{w,f}$ is frequently hypercyclic, frequently recurrent or Devaney chaotic on $L^p(\mu)$ 
if and only if $C_{f}$ is frequently hypercyclic, frequently recurrent or Devaney chaotic on $L^p(\nu)$, respectively.
On the other hand,  by \cite[Theorem~2.9]{AniMaiSeo24}, $C_{f}$ is frequently recurrent on $L^p(\nu)$ 
if and only if $C_{f}$ is frequently hypercyclic on $L^p(\nu)$.
Moreover,  by \cite[Theorem 3.7]{UDarBPir21}, $C_{f}$ is frequently hypercyclic on $L^p(\nu)$ 
if and only if $C_{f}$ is Devaney chaotic on $L^p(\nu)$.
\end{proof}


\section*{Declarations}

\noindent
{\bf Conflict of Interest} \ The authors declare that they have no conflict of interest.

\medskip\noindent
{\bf Data Availability} \ Not applicable.


\section*{Acknowledgements}

The first author was partially supported by CNPq -- Project {\#}308238/2021-4, and by CAPES -- Finance Code 001.
The first and the second authors were partially supported by Project PID2022-139449NB-I00, 
funded by MCIN/AEI/10.13039/501100011033/FEDER, UE.


\bigskip\bigskip\noindent
{\sc Nilson C. Bernardes Jr.}

\medskip\noindent
Departamento de Matem\'atica Aplicada, Instituto de Matem\'atica, Universidade Federal do Rio de Janeiro, 
Caixa Postal 68530, RJ 21941-909, Brazil.\\
\textit{ e-mail address}: bernardes@im.ufrj.br

\bigskip\medskip\noindent {\sc Antonio Bonilla}

\medskip\noindent
Departamento de An\'alisis Matem\'atico and Instituto de Matem\'aticas y Aplicaciones (IMAULL), 
Universidad de La Laguna, C/Astrof\'{\i}sico Francisco S\'anchez, s/n, 38721 La Laguna, Tenerife, Spain\\
\textit{ email address:} abonilla@ull.edu.es

\bigskip\medskip\noindent
{\sc Jo\~ao V. A. Pinto}

\medskip\noindent
Departamento de Matem\'atica Aplicada, Instituto de Matem\'atica, Universidade Federal do Rio de Janeiro, 
Caixa Postal 68530, RJ 21941-909, Brazil.\\
\textit{ e-mail address}: joao.pinto@ufu.br

\end{document}